\documentclass[reqno, 12pt]{amsart}
\usepackage{amsmath,amsxtra,amssymb,latexsym, amscd,amsthm}
\usepackage[utf8]{inputenc}
\usepackage[mathscr]{euscript}
\usepackage{mathrsfs}
\usepackage{enumerate}
\usepackage{color}
\usepackage{cite}

\newtheorem {theorem}{Theorem}[section]

\newtheorem {lemma}{Lemma}[section]

\theoremstyle{definition}
\newtheorem{definition}{Definition}[section]

\newtheorem{remark}{Remark}[section]

\def\ees{{\accent"5E e}\kern-.385em\raise.2ex\hbox{\char'23}\kern-.08em}
\def\EES{{\accent"5E e}\kern-.5em\raise.8ex\hbox{\char'23 }}
\def\ow{o\kern-.42em\raise.82ex\hbox{\vrule width .12em height .0ex depth .075ex \kern-0.16em \char'56}\kern-.07em}
\def\OW{o\kern-.460em\raise1.36ex\hbox{
\vrule width .13em height .0ex depth .075ex \kern-0.16em
\char'56}\kern-.07em}
\def\DD{D\kern-.7em\raise0.4ex\hbox{\char '55}\kern.33em}
\title[]{Nichtnegativstellens\"atze for definable functions in o-minimal structures}

\author{S\~i Ti\d{\^e}p \DD inh$^\dagger$}
\address{Institute of Mathematics, VAST, 18, Hoang Quoc Viet Road, Cau Giay District 10307, Hanoi, Vietnam}
\email{dstiep@math.ac.vn}

\author{TI\EES N-S\OW n Ph\d{a}m$^\ddagger$}
\address{Department of Mathematics, Dalat University, 1 Phu Dong Thien Vuong, Dalat, Vietnam}
\email{sonpt@dlu.edu.vn}

\date{ \today}



\begin{document}

\begin{abstract} 
This paper addresses to Nichtnegativstellens\"atze for definable functions in o-minimal structures on $(\mathbb{R}, +, \cdot).$ Namely, let 
$f, g_1, \ldots, g_l  \colon \mathbb{R}^n \to \mathbb{R}$ be definable $C^p$-functions ($p \ge 2$) and assume that $f$ is non-negative on $S := \{x \in \mathbb{R}^n \ | \ g_1(x) \ge 0, \ldots,  g_l(x) \ge  0 \}.$ Under some natural hypotheses on
zeros of $f$ in $S,$ we show that $f$ is expressible in the form $f = \phi_0 + \sum_{i = 1}^l \phi_i g_i,$ where each $\phi_i$ is a sum of squares of definable $C^{p - 2}$-functions. As a consequence, we derive global optimality conditions which generalize the Karush--Kuhn--Tucker optimality conditions for nonlinear optimization.
\end{abstract}

\maketitle

\section{Introduction}

A classical Positivstellens\"atz proved by Krivine~\cite{Krivine1964}, and independently by Stengle~\cite{Stengle1974}, states that a polynomial $f$ is non-negative over a basic closed semi-algebraic set 
$$S := \{x \in \mathbb{R}^n \ | \ g_1(x) \ge 0, \ldots, g_l(x) \ge 0\}$$ 
if and only if there exist an integer number $d \ge 0$ and polynomials $\psi, \phi$  in the preordering  generated by $g_1, \ldots, g_l$ over the sums of squares (of polynomials) such that
$$\psi f = \phi  + f^{2d}.$$
Note that, the denominator $\psi$ cannot be omitted (see \cite{Bochnak1998, Marshall2008}).

Schm\"udgen~\cite{Schmudgen1991} showed that if $S$ is compact and $f$ is strictly positive on $S,$ then no denominators are needed; that is, $\psi$ can be chosen as $1$ in the above expression. Moreover, under some more restrictive hypotheses on the $g_i,$ Putinar~\cite{Putinar1993} proved that the polynomial $f$ can be represented as
$$f = \phi_0 + \sum_{i = 1}^l \phi_i g_i,$$ 
where each $\phi_i$ is a sum of squares; that is, $f$ belongs to the quadratic module generated by the $g_i$'s over the sums of squares, rather than the preordering generated by them. With additional conditions on zeros of $f$ in $S,$ Scheiderer~\cite{Scheiderer2005} and Marshall~\cite{Marshall2006} showed that the above equation remains true if we replace the assumption ``$f$ is strictly positive on $S$'' by ``$f$ is non-negative on $S$''.

In \cite{Acquistapace2002} (see also \cite{Acquistapace2000, Acquistapace2009, Fischer2011, Fischer2013}),  Acquistapace, Andradas and Broglia established Positivstellens\"atze for differentiable functions in o-minimal structures. In particular, they proved that a function $f$ that is non-negative on a closed set (not necessarily compact)
$$S = \{x \in \mathbb{R}^n \ | \ g_1(x) \ge 0, \ldots, g_l(x) \ge 0\}$$ 
admits a representation of the form
$$\psi^2 f = \phi_0^2 + \sum_{i = 1}^l \phi_i^2 g_i.$$ 
Again, denominators are necessary. Indeed, for any $S$ with non-empty interior, there are definable functions that are non-negative over $S$ and do not belong to the preordering generated by $g_1, \ldots, g_l$ over the sums of squares; that is, the denominator $\psi$ in the above equation cannot be omitted; for more details, see \cite[Remark~3.9]{Acquistapace2002}.

This paper deals with Nichtnegativstellens\"atze (local and global) for definable functions of class $C^p$ ($p \ge 2$) in o-minimal structures on $(\mathbb{R},+,\cdot)$ without any compactness assumption. 
Indeed, let $f$ be a definable function, which is non-negative on a basic definable set 
$$S = \{x \in \mathbb{R}^n \ | \ g_1(x) \ge 0, \ldots, g_l(x) \ge 0\}.$$ 
We give natural sufficient conditions in terms of the first and second derivatives of $f$ at its zeros in $S,$ so that $f$ can be represented as
$$f = \phi_0 + \sum_{i = 1}^l \phi_i g_i,$$ 
where each $\phi_i$ is a sum of squares. Our proof is elementary, using Morse's lemma and partitions of unity. 

As a consequence, we obtain global optimality conditions which generalize the Karush--Kuhn--Tucker optimality conditions for nonlinear optimization.

We finish this section by noting that all the statements and proofs in this paper remain true if we remove in them the term ``definable''. Also, all the statements still hold if we replace $\mathbb{R}^n$ by any real (definable) manifold; however, to lighten the exposition, we do not pursue this extension here.

The rest of this paper is organized as follows. Section~\ref{SectionPreliminary} contains some properties of definable sets and functions in o-minimal structures. For the convenience of the reader, local optimality conditions in nonlinear programming theory are also recalled here. 
Nichtnegativstellens\"atze for definable functions (Theorems~\ref{Nichtnegativstellensatz-1} and~\ref{Nichtnegativstellensatz-2}) are established in Section~\ref{Section3}. Finally, global optimality conditions (Theorem~\ref{GlobalOptimalityConditions}) are presented in Section~\ref{Section4}.

\section{Preliminaries} \label{SectionPreliminary}

In this paper,  we deal with the Euclidean space $\mathbb{R}^n$ equipped with the usual scalar product $\langle \cdot, \cdot \rangle$ and the corresponding Euclidean norm $\| \cdot\|.$ Let $\mathbb{R}_{+}$ denote the set of positive real numbers. 
If $f$ is a function in $x,$ $\nabla f(x)$ (resp., $\nabla^2 f(x)$) denotes the gradient vector (resp., Hessian matrix) of $f$ at $x.$

\subsection{O-minimal structures}

The notion of o-minimality was developed in the late 1980s after it was noticed that many proofs of analytic and geometric properties of semi-algebraic sets and mappings can be carried over verbatim for sub-analytic sets and mappings. The reader is referred to \cite{Coste2000, Dries1998, Dries1996} for more details.

\begin{definition}{\rm
An {\em o-minimal structure} on $(\mathbb{R}, +, \cdot)$ is a sequence $\mathcal{D} := (\mathcal{D}_n)_{n \in \mathbb{N}}$ such that for each $n \in \mathbb{N}$:
\begin{itemize}
\item [(a)] $\mathcal{D}_n$ is a Boolean algebra of subsets of $\mathbb{R}^n$.
\item [(b)] If $X \in \mathcal{D}_m$ and $Y \in \mathcal{D}_n$, then $X \times Y \in \mathcal{D}_{m+n}.$
\item [(c)] If $X \in \mathcal{D}_{n + 1},$ then $\pi(X) \in \mathcal{D}_n,$ where $\pi \colon \mathbb{R}^{n+1} \to \mathbb{R}^n$ is the projection on the first $n$ coordinates.
\item [(d)] $\mathcal{D}_n$ contains all algebraic subsets of $\mathbb{R}^n.$
\item [(e)] Each set belonging to $\mathcal{D}_1$ is a finite union of points and intervals.
\end{itemize}
}\end{definition}

A set belonging to $\mathcal{D}$ is said to {\em definable} (in that structure). {\em Definable mappings} in structure $\mathcal{D}$ are mappings whose graphs are definable sets in $\mathcal{D}.$

Examples of o-minimal structures are
\begin{itemize}
\item the semi-algebraic sets (by the Tarski--Seidenberg theorem),
\item the globally sub-analytic sets, i.e., the sub-analytic sets of $\mathbb{R}^n$ whose (compact) closures in the real projective space $\mathbb{R}\mathbb{P}^n$ are sub-analytic (using Gabrielov's complement theorem).
\end{itemize}

In this paper, we fix an o-minimal structure on $(\mathbb{R}, +, \cdot).$ The term ``definable'' means definable in this structure.
We recall some useful facts which we shall need later.

\begin{lemma}\label{ConnectedComponents}
Every definable set has a finite number of connected components.
\end{lemma}
\begin{proof}
See \cite[Properties~4.3]{Dries1996}.
\end{proof}

\begin{lemma} \label{MillerLemma}
Let $U$ be a definable open subset of $\mathbb{R}^n$ containing $0$ and $f \colon U \to \mathbb{R}$ be a definable $C^p$-function $(p \ge 1)$ with $f(0) = 0.$ Then there are definable $C^{p - 1}$-functions $f_i \colon U \to \mathbb{R}$ such that on $U$ we have 
$$f = x_1 f_1 + \cdots + x_n f_n.$$
\end{lemma}
\begin{proof}
See \cite[Lemma~A.6]{Peterzil2007}.
\end{proof}

To simplify notation in what follows, the notation ``$F \colon (\mathbb{R}^n, x^*) \to (\mathbb{R}^m, y^*)$'' means that $F$ is a mapping from a definable open neighbourhood of $x^* \in \mathbb{R}^n$ into $\mathbb{R}^m$ with $F(x^*) = y^*.$ 

\begin{lemma} [Morse's lemma] \label{MorseLemma}
Let $U$ be a definable open subset of $\mathbb{R}^n$ containing $0$ and $f \colon U \to \mathbb{R}$ be a definable $C^{p}$-function $(p \ge 2)$. If $0$ is a non-degenerate critical point of $f$ (i.e., $\nabla f(0) = 0$ and the Hessian matrix $\nabla^2 f(0)$ of $f$ at $0$ is non-singular), then there is a definable $C^{p - 2}$-diffeomorphism $\Phi \colon (\mathbb{R}^n, 0) \to (\mathbb{R}^n, 0)$ such that
$$f \circ \Phi(y) = f(0) - y_1^2 - \cdots - y_{\ell}^2 +  y_{\ell + 1}^2 + \cdots + y_n^2,$$
where $\ell$ is the number of negative eigenvalues (multiplicity taken into account) of $\nabla^2 f(0).$
\end{lemma}
\begin{proof}
Confer \cite[Lemma~A.7]{Peterzil2007} (see also \cite[Section~6.1]{Hirsch1994} and \cite[Theorem~2.8.2]{Jongen2000}).
\end{proof}
 
\begin{lemma} [Definable partition of unity] \label{PartitionUnity}
Let $\{U_k\}_{k = 1, \ldots, K}$ be a finite definable open covering of $\mathbb{R}^n.$ 
Then for any $p \ge 0$, there exist definable $C^p$-functions $\theta_k \colon \mathbb{R}^n \to [0, 1], k = 1, \ldots, K,$ such that the following statements hold:
\begin{enumerate}[{\rm (i)}]
\item $\mathrm{supp\,}\theta_k \subset U_k,$  where $\mathrm{supp\,} \theta_k$ denotes the closure of the set $\{x \in \mathbb{R}^n \ | \ \theta_k(x) \ne 0\};$
\item $\sum_{k = 1}^K [\theta_k(x)]^2 = 1$ for all $x \in \mathbb{R}^n.$
\end{enumerate}
 \end{lemma} 

\begin{proof}
It is well-known that (see, for example, \cite[Theorem~3.4.2]{Escribano2000}, \cite[Theorem~2.1]{Hirsch1994}, \cite[Lemma~3.7]{Dries1998}), there exists a definable partition of unity $\{\phi_k\}_{k = 1, \ldots, K}$ subordinated to the covering $\{U_k\}_{k = 1, \ldots, K}.$ Clearly, the functions $\theta_k := \frac{\phi_k}{\sqrt{\sum_{k = 1}^K \phi_k^2}}$ have the desired properties.
\end{proof}

\subsection{Optimality conditions}
We give here a short review of optimality conditions in nonlinear programming theory (confer \cite[Section~4.3]{Bertsekas2016}).

Let $f, g_1, \ldots, g_l, h_1, \ldots, h_m  \colon \mathbb{R}^n \to \mathbb{R}$ be $C^p$-functions ($p \ge 2$) and assume that
$$S := \{x \in \mathbb{R}^n \ | \ g_1(x) \ge 0, \ldots,  g_l(x) \ge  0, h_1(x)  = 0, \ldots, h_m(x)  = 0 \} \ne \emptyset.$$

\begin{definition} {\rm 
The constraint set $S$ is said to be {\em regular at $x \in S$} if the gradient vectors $\nabla g_i(x), \ i \in I(x)$ and $\nabla h_j(x), \ j = 1, \ldots, m,$ are linearly independent, where 
$$I(x) := \{i \in \{1, \ldots, l\} \ | \ g_i(x) = 0 \}$$
is called the {\em set of active constraint indices at $x.$} The set $S$ is called {\em regular} if it is regular at every point $x \in S.$
}\end{definition}

Let $x^*$ be a local minimizer of the restriction of $f$ on $S$ and assume that $S$ is regular at $x^*.$ It is well-known that there exist (unique) Lagrange multipliers $\lambda_i, i = 1, \ldots, l,$ and $\nu_j, j = 1, \ldots, m,$ satisfying the {\em Karush--Kuhn--Tucker optimality conditions} (KKT~optimality conditions for short)
\begin{eqnarray*}
&&  \nabla  f(x^*)  - \sum_{i = 1}^{l} \lambda_i \nabla g_i({x^*}) - \sum_{j = 1}^{m} \nu_j \nabla  h_j({x^*})  = 0, \\
&&  \lambda_i g_i(x^*) = 0, \ \lambda_i \ge 0, \ \textrm{ for } \ i = 1, \ldots, l.
\end{eqnarray*}
Recall that the {\em strict complementarity condition} holds at $x^*$ if it holds that
$$\lambda_1 + g_1(x^*) > 0, \ldots, \lambda_l + g_l(x^*) > 0.$$
Note that strict complementarity is equivalent to $\lambda_i > 0$ for every $i \in I(x^*).$

Let ${L}(x)$ be the associated Lagrangian function
$${L}(x) := f({x}) - \sum_{i \in I(x^*)} \lambda_i g_i({x}) - \sum_{j = 1}^m \nu_j h_j({x}),$$
where $I(x^*)$ is the set of active constraint indices at $x^*.$ Then the {\em second-order necessity condition} holds at $x^*,$ that is
$$v^T \nabla^2 L(x^*) v \ge  0 \quad \textrm{ for all } \quad v \in T_{x^*}S.$$
Here $ \nabla^2 L(x^*)$ is the Hessian matrix of $L$ at $x^*$ and $T_{x^*}S$ stands for the {\em (generalized) tangent space} of $S$ at $x^*$:
$$T_{x^*}S := \left\{
\begin{array}{lll}
v \in \mathbb{R}^n & : & \langle v, \nabla g_{i}(x^*) \rangle = 0, \ i \in I(x^*) \ \textrm{ and } \\
&& \langle v, \nabla h_{j}(x^*) \rangle = 0, \ j = 1, \ldots, m 
\end{array}
\right\}.
$$
If it holds that
$$v^T \nabla^2 L(x^*) v >  0 \quad \textrm{ for all } \quad v \in T_{x^*}S\setminus\{ 0\},$$
we say the {\em second-order sufficiency condition} holds at $x^*.$

\begin{remark}
(i) Let $x^* \in S$ be a local minimizer of $f$ on $S$ and assume that $S$ is regular at $x^*.$ According to \cite[Lemma~3.2.16]{Jongen2000}, the point $x^*$ is a {\em (generalized)  non-degenerate critical point} of the restriction of $f$ to $S$ if and only if the strict complementarity and second-order sufficiency conditions hold at $x^*.$

(ii) Using transversality arguments, one can show that the regularity, strict complementarity and second order sufficiency conditions hold generically. For related works, see \cite{HaHV2017, Nie2014}.
\end{remark}

\section{Nichtnegativstellens\"atze for definable functions} \label{Section3}

In this section we prove two Nichtnegativstellens\"atze for definable functions. So let $f, g_1, \ldots, g_l, h_1, \ldots, h_m  \colon \mathbb{R}^n \to \mathbb{R}$ be definable $C^p$-functions ($p \ge 2$) and assume that
$$S := \{x \in  \mathbb{R}^n \ | \ g_1(x) \ge 0, \ldots,  g_l(x) \ge  0, h_1(x)  = 0, \ldots, h_m(x)  = 0 \} \ne \emptyset.$$
The first main result of the paper reads as follows.

\begin{theorem}[Local Nichtnegativstellens\"atz] \label{Nichtnegativstellensatz-1}
Assume that $f$ is non-negative on $S$ and let $x^* \in S$ be a zero of $f$ in $S.$ If $S$ is regular at $x^*$ and the strict complementarity and second-order sufficiency conditions hold at $x^*,$ then there are a definable open neighbourhood $U$ of $x^*$ and definable $C^{p - 2}$-functions $\phi_i, \psi_j \colon U \rightarrow \mathbb{R}$ for $i = 0, \ldots, l$ and $j = 1, \ldots, m,$ where each $\phi_i$ is a sum of squares of definable $C^{p - 2}$-functions, such that on $U$ we have
\begin{eqnarray*}
f  &=& \phi_0 + \sum_{i = 1}^l \phi_i g_i + \sum_{j = 1}^m \psi_j h_j.
\end{eqnarray*}
\end{theorem}
\begin{proof}
(cf. \cite{Jongen2000, Nie2014}). For convenience, we can generally assume $x^* = 0,$ up to a shifting.

Recall that $I(x^*) := \{ i \in \{1, \ldots, l\} \ | \ g_i(x^*) = 0\}$ is the index set of inequality constraints that are active at $x^*.$ 
By renumbering, we may assume that $I(x^*) := \{1, \ldots, k\}.$ Since $S$ is regular at $0,$ the gradient vectors
$$\nabla g_{1}(0), \ldots, \nabla g_{k}(0), \nabla h_{1}(0), \ldots, \nabla h_{m}(0)$$
are linearly independent. Let $d := n - m - k.$ Up to a linear coordinate transformation, we can further assume that
$$\begin{array}{lcllcllcl}
\nabla g_{1}(0) &=& e^{d + 1}, & \ldots, & \nabla g_{k}(0) & = &  e^{d + k},\\
\nabla h_{1}(0) &=& e^{d + k + 1}, & \ldots, & \nabla h_{m}(0)  & = &  e^{n},\\
\end{array}$$
where $e^1, \ldots, e^n$ are the canonical basis vectors in $\mathbb{R}^n.$ Note that the space $T_{x^*}S$ is determined by the vectors
$e^{1}, \ldots, e^d.$

Define the definable $C^p$-mapping $\Phi \colon \mathbb{R}^n \longrightarrow \mathbb{R}^n, x \mapsto \Phi(x),$ by 
$$\Phi(x) :=  (x_1, \ldots, x_d, g_{1}(x), \ldots, g_{k}(x), h_{1}(x), \ldots, h_{m}(x)).$$
Clearly, $\Phi(0) = 0$ and the Jacobian matrix $D\Phi(0)$ of $\Phi$ at $0$ is the identity matrix $I_n.$ Thus, by the inverse function theorem, $\Phi$ is 
a local $C^p$-diffeomorphism in some neighbourhood of $0$ with the inverse
$$\Phi^{-1} \colon (\mathbb{R}^n, 0) \to (\mathbb{R}^n, 0), \quad t \mapsto x := \Phi^{-1}(t).$$ 
So, $t := (t_1, \ldots ,t_n)$ can serve as a coordinate system for $\mathbb{R}^n$ around $0.$ In the $t$-coordinate system and in a neighborhood of $0,$ the set $S$ defined by 
$$t_{d + 1} \geq 0, \ldots, t_{d + k} \geq 0, t_{d + k + 1}  = 0, \ldots, t_n = 0.$$

Let $\lambda_i$ and $\nu_j$  be the Lagrange multipliers with respect to the minimizer $x^*.$ Define the Lagrangian function
$$L(x):= f(x) - \sum_{i \in I(x^*)} \lambda_{i} g_{i} (x) - \sum_{j = 1}^m \nu_{j}  h_{j}(x).$$
Note that $\nabla L(0)=0.$ In the $t$-coordinate system, define the functions
\begin{eqnarray*}
F(t) &:=& f(\Phi^{-1} (t)), \\
\widehat{L} (t) &:=&  L(\Phi^{-1} (t) ) = F(t) - \sum_{r = d + 1}^{d + k} \lambda_{{r - d}} t_r - \sum_{r = d + k + 1}^n \nu_{r - d - k} t_r.
\end{eqnarray*}
Clearly, 
$$\nabla \widehat{L} (0) = \nabla L(0) D \Phi^{-1}(0) =  \nabla L(0)  = 0.$$ 
This implies that
$$ \frac{\partial F }{\partial t_r}(0) = 
 \begin{cases}
 0 & \textrm{ if }   r = 1, \ldots , d, \\
 \lambda_{{r - d}} & \textrm{ if }   r = d +1 , \ldots , d + k, \\
 \nu_{r - d - k} &  \textrm{ if }  r = d + k + 1, \ldots , n. \\
\end{cases}$$
Furthermore, for $(t_{1}, \ldots, t_d)$ near $0 \in \mathbb{R}^d,$ it holds that
\begin{eqnarray*}
F(t_{1}, \ldots, t_d, 0, \ldots , 0) & =& \widehat{L} (t_{1}, \ldots, t_d, 0, \ldots , 0) , \\
& = & L(\Phi^{-1}(t_{1}, \ldots, t_d, 0, \ldots , 0) ).
\end{eqnarray*}
Let $x(t) := \Phi^{-1}(t)=(\Phi_1^{-1} (t), \ldots , \Phi_n^{-1} (t)).$ For all $i, j$, we have
$$\frac{\partial^2 \widehat{L}(t)}{\partial t_i \partial t_j} = \sum_{1 \leq r, s  \leq n  }  \frac{ \partial^2 L(x(t)) }{\partial x_r \partial x_s } .  \frac{\partial \Phi_{r}^{-1} (t) }{\partial t_i }.  \frac{\partial \Phi_{s}^{-1} (t) }{\partial t_j }  + \sum_{1 \leq r  \leq n  }  \frac{ \partial L(x(t)) }{\partial x_r  } .  \frac{\partial^2 \Phi_{r}^{-1} (t) }{\partial t_i  \partial t_j }. $$
Note that $\nabla L(0)=0$ and $x(0) = \Phi^{-1}(0) = 0.$ Hence
$$\frac{ \partial^2 \widehat{L} }{\partial t_i \partial t_j }(0)  = \sum_{1 \leq r, s  \leq n  }  \frac{ \partial^2 L }{\partial x_r \partial x_s }(0) .  \frac{\partial \Phi_{r}^{-1} }{\partial t_i }(0). \frac{\partial \Phi_{s}^{-1}  }{\partial t_j }(0).$$
On the other hand, we have $D \Phi(0) = D \Phi^{-1} (0) = I_n$-the identity matrix. Therefore for all  $i, j = 1, \ldots, d,$ 
$$\frac{\partial^2 F}{\partial t_i \partial t_j}  \bigg \arrowvert_{t = 0} =
\frac{ \partial^2 \widehat{L}}{\partial t_i \partial t_j }\bigg \arrowvert_{t = 0} 
= \frac{\partial^2 L}{\partial  x_i \partial x_j}  \bigg\arrowvert_{x=0}.$$
Since the vector space $T_{x^*}S$ is defined by the vectors $e^{1}, \ldots, e^d,$ the second-order sufficiency condition implies that the sub-Hessian
$$\left( \frac{ \partial^2 {L}  }{\partial x_i \partial x_j }(0) \right)_{1 \leq i, j \leq d}$$
is positive definite. 

Define the definable $C^p$-function $A \colon (\mathbb{R}^d, 0) \to (\mathbb{R}, 0),$ by 
$$A(t_1, \ldots, t_d) := F(t_1, \ldots, t_d, 0, \ldots, 0).$$
Then $A(0) = 0, \nabla A(0) = 0$ and the Hessian matrix $\nabla^2 A(0)$ is positive define. On the other hand, by Lemma~\ref{MillerLemma}, there exist definable $C^{p - 1}$-functions $B_r \colon (\mathbb{R}^{d + k}, 0) \to (\mathbb{R}, 0)$ for $r = 1, \ldots, k,$ and $C_r \colon (\mathbb{R}^{n}, 0) \to (\mathbb{R}, 0)$ for $r = 1, \ldots, m$ such that
\begin{eqnarray*}
F(t_1, \ldots, t_{d + k}, 0, \ldots, 0) &=& F(t_1, \ldots, t_d, 0, \ldots, 0) + \sum_{r = d + 1}^{d + k} t_r B_{r - d} (t_1, \ldots, t_{d + k}), \\
F(t_1, \ldots, t_n) &=& F(t_1, \ldots, t_{d + k}, 0, \ldots, 0) + \sum_{r = d + k + 1}^{n} t_r C_{r - d - k} (t_1, \ldots, t_n).
\end{eqnarray*}
Then for all $t := (t_1, \ldots, t_n) \in \mathbb{R}^n$ near $0,$ we have
\begin{eqnarray*}
F(t_1, \ldots, t_n) &=& A(t_1, \ldots, t_d) + \sum_{r = d + 1}^{d + k} t_r B_{r - d} (t_1, \ldots, t_{d + k}) + \sum_{r = d + k + 1}^{n} t_r C_{r - d - k} (t_1, \ldots, t_n).
\end{eqnarray*}

In view of Lemma~\ref{MorseLemma} (applied to the function $A$), there is a definable $C^{p - 2}$-diffeomorphism 
$$\theta \colon (\mathbb{R}^{d}, 0) \to (\mathbb{R}^{d}, 0), \quad (t_1, \ldots, t_d) \mapsto (z_1, \ldots, z_d),$$ 
such that
\begin{eqnarray*}
A\circ \theta^{-1} (z_1, \ldots, z_d) &=& \sum_{r = 1}^d z_r^2.
\end{eqnarray*}
We extend $\theta$ to a definable $C^{p - 2}$-diffeomorphism 
$$\Theta \colon (\mathbb{R}^{n}, 0) \to (\mathbb{R}^{n}, 0), \quad t \mapsto z := \Theta(t),$$
by putting
\begin{eqnarray*}
\Theta(t_1, \ldots, t_n) &:=& (\theta(t_1, \ldots, t_d), t_{d + 1}, \ldots, t_n).
\end{eqnarray*}
Next we put
\begin{eqnarray*}
\widetilde{B}_i (z_1, \ldots, z_{d + k}) &:=& B_i(\theta^{-1}(z_1, \ldots, z_d), z_{d + 1}, \ldots, z_{d + k}), \quad i = 1, \ldots, k, \\
\widetilde{C}_j (z_1, \ldots, z_n) &:=& C_j(\theta^{-1}(z_1, \ldots, z_d), z_{d + 1}, \ldots, z_{n}), \quad j = 1, \ldots, m.
\end{eqnarray*}
Clearly, $\widetilde{B}_i$ and $\widetilde{C}_j$ are definable $C^{p - 2}$-functions, $\widetilde{B}_i (0) = B_i(0) = \lambda_{{i}} > 0.$ 
In particular, in some neighbourhood of $0 \in \mathbb{R}^{d + k},$ the functions $\widetilde{B}_i$ are squares of definable $C^{p - 1}$-functions.
For all $z := (z_1, \ldots, z_n) \in \mathbb{R}^n$ near $0,$ we have
\begin{eqnarray*}
f \circ (\Theta \circ \Phi) ^{-1}(z)
&=& F \circ \Theta^{-1}(z) \\
&=& \sum_{r = 1}^d z_r^2 + \sum_{r = d + 1}^{d + k} z_r \widetilde{B}_{r - d}(z_1, \ldots, z_{d + k}) + 
\sum_{r = d + k + 1}^{n} z_r \widetilde{C}_{r - d - k}(z).
\end{eqnarray*}

Let $x := (\Theta \circ \Phi) ^{-1}(z);$ or equivalently, 
$$z = (\theta(x_1, \ldots, x_d), g_{1}(x), \ldots, g_{k}(x), h_{1}(x), \ldots, h_{m}(x)).$$
Then it is easy to see that the functions
\begin{eqnarray*}
\phi_0(x) &:=& \|\theta(x_1, \ldots, x_d)\|^2, \\
\phi_i(x) &:=& \widetilde{B}_{i}\big(\theta(x_1, \ldots, x_d), g_{1}(x), \ldots, g_{k}(x)), \quad i = 1, \ldots, k, \\
\phi_i(x) &:=& 0, \quad i = k + 1, \ldots, l, \\
\psi_j(x) &:=& \widetilde{C}_{j}\big(\theta(x_1, \ldots, x_d), g_{1}(x), \ldots, g_{k}(x), h_1(x), \ldots, h_m(x) \big), \quad j = 1, \ldots, m,
\end{eqnarray*}
have the desired properties.
\end{proof}

We are now ready to prove our second main result.
\begin{theorem}[Global Nichtnegativstellens\"atz] \label{Nichtnegativstellensatz-2}
Assume that $f$ is nonnegative on $S.$ If the regularity, strict complementarity and second-order sufficiency conditions hold at every zeros of the restriction of $f$ on $S,$ then there are definable $C^{p - 2}$-functions $\phi_i, \psi_j \colon \mathbb{R}^n \rightarrow \mathbb{R}$ for $i = 0, \ldots, l$ and $j = 1, \ldots, m,$ where each $\phi_i$ is a sum of squares of definable $C^{p - 2}$-functions, such that
$$f = \phi_0 + \sum_{i = 1}^l \phi_i g_i + \sum_{j = 1}^m \psi_j h_j.$$
\end{theorem}

\begin{proof}
Our assumptions yield that the definable set $f^{-1}(0) \cap S$ is discrete (see, for example, \cite[Corollary~3.2.30]{Jongen2000}). This, together with Lemma~\ref{ConnectedComponents}, implies that $f^{-1}(0) \cap S$ is a finite set, say $\{x^*_1, \ldots, x^*_N\}.$ In view of Theorem~\ref{Nichtnegativstellensatz-1}, for each $k = 1, \ldots, N,$ there exist a definable open neighbourhood $U_k$ of $x^*_k$ and definable $C^{p - 2}$-functions $\phi_{k, i}, \psi_{k, j} \colon U_k \to \mathbb{R}$ for $i = 0, 1, \ldots, l$ and $j = 1, \ldots, m,$ where each $\phi_{k, i}$ is a sum of squares of definable $C^{p - 2}$-functions, such that on $U_k$ we have
$$f = \phi_{k, 0} + \sum_{i = 1}^l \phi_{k, i} g_i + \sum_{j = 1}^m \psi_{k, j} h_j.$$

Since $f$ is positive on the set $S \setminus \{x^*_1, \ldots, x^*_N\},$ we can find a (definable) open set $U_{N + 1}$ containing $S \setminus \bigcup_{k = 1}^N U_k$ such that $f$ is positive on $U_{N + 1}.$ On the set $U_{N + 1},$ define the definable $C^p$-functions $\phi_{k, i}$ and $\psi_{k, j}$ by
\begin{eqnarray*}
\phi_{N+1, 0} &:=& f, \\
\phi_{N+1, i} &:=& 0  \quad \textrm{ for } \quad  i = 1, \ldots, l, \\
\psi_{N+1, j} &:=& 0 \quad \textrm{ for } \quad j = 1, \ldots, m.
\end{eqnarray*}
For $k = N + 2, \ldots, N + l + 1,$ on the definable open set $U_k := \{x \in \mathbb{R}^n \ | \ g_{k - N - 1}(x) < 0\},$ define the definable $C^p$-functions $\phi_{k, i}$ and $\psi_{k, j}$ by
\begin{eqnarray*}
\phi_{k, 0} &:=& \left( \frac{f + 1}{2} \right)^2, \\
\phi_{k, i} &:=& \left( \frac{f - 1}{2} \right)^2 \cdot \frac{1}{l(-g_i)}  \quad \textrm{ for } \quad  i = 1, \ldots, l, \\
\psi_{k, j} &:=& 0 \quad \textrm{ for } \quad j = 1, \ldots, m.
\end{eqnarray*}
For $k = N + l + 2, \ldots, N + l + m + 1,$ on the definable open set $U_k :=  \{x \in \mathbb{R}^n \ | \ h_{k - N - l - 1}(x) \ne 0\},$ define the definable $C^p$-functions $\phi_{k, i}$ and $\psi_{k, j}$ by
\begin{eqnarray*}
\phi_{k, 0} &:=& 0 \quad \textrm{ for } \quad i = 0, \ldots, l, \\
\psi_{k, j} &:=& \frac{f}{m h_j} \quad \textrm{ for } \quad  j = 1, \ldots, m.
\end{eqnarray*}

By definition, on the set $U_k,$ $k = N + 1, \ldots, N + l + m + 1,$ we have
$$f = \phi_{k, 0} + \sum_{i = 1}^l \phi_{k, i} g_i + \sum_{j = 1}^m \psi_{k, j} h_j.$$
Since $\{U_k\}_{k = 1, \ldots, N + l + m + 1}$ is a the family of definable open sets covering $\mathbb{R}^n,$ it follows from Lemma~\ref{PartitionUnity} that there exist definable $C^p$-functions $\theta_k \colon \mathbb{R}^n \to [0, 1],$ $k = 1, \ldots, N + l + m + 1,$ such that $\mathrm{supp\,} \theta_k \subset U_k$ and
$$\sum_{k = 1}^{N + l + m + 1} \big[\theta_k(x) \big]^2 = 1 \quad \textrm{ for all } \quad x \in \mathbb{R}^n.$$
For $k = 1, \ldots, N + l + m + 1,$ $i = 0, 1, \ldots, l,$ and $j = 1, \ldots, m,$ the functions $\theta_k^2 \phi_{k, i}$ and $\theta_k^2 \psi_{k, j}$ extend by $0$ to (definable) $C^{p - 2}$-functions over $\mathbb{R}^n.$ By construction, the functions $\theta_k^2 \phi_{k, i}$ are sums of squares of definable $C^{p - 2}$-functions.

Finally, on $\mathbb{R}^n$ we have
\begin{eqnarray*}
f 
&=& 1 \cdot f \  = \  \left(\sum_{k = 1}^{N + l + m + 1} \theta_k^2 \right) f  \ = \  \sum_{k = 1}^{N + l + m + 1} \theta_k^2 f 
\\
&=& \sum_{k = 1}^{N + l + m + 1} \left (\theta_k^2 \phi_{k, 0} + \sum_{i = 1}^l \theta_k^2  \phi_{k, i} g_i + \sum_{j = 1}^m \theta_k^2  \psi_{k, j} h_j \right) \\
&=& \sum_{k = 1}^{N + l + m + 1} \theta_k^2 \phi_{k, 0} + \sum_{i = 1}^l \left( \sum_{k = 1}^{N + l + m + 1} \theta_k^2  \phi_{k, i} \right) g_i + \sum_{j = 1}^m 
\left ( \sum_{k = 1}^{N + l + m + 1} \theta_k^2  \psi_{k, j} \right)  h_j.
\end{eqnarray*}
The proof is complete.
\end{proof}

\section{Global optimality conditions}\label{Section4}

In this section, we derive global optimality conditions which generalize the KKT~optimality conditions for nonlinear optimization. 

Let $f, g_1, \ldots, g_l, h_1, \ldots, h_m  \colon \mathbb{R}^n \to \mathbb{R}$ be definable $C^p$-functions ($p \ge 3$) and assume that
$$S := \{x \in \mathbb{R}^n \ | \ g_1(x) \ge 0, \ldots,  g_l(x) \ge  0, h_1(x)  = 0, \ldots, h_m(x)  = 0 \} \ne \emptyset.$$
Let $(x^*, \lambda, \nu) \in S \times \mathbb{R}_+^l \times \mathbb{R}^m$ be a vector satisfying the KKT~optimality conditions associated  with the problem $\min_{x \in S} f(x),$ that is
\begin{eqnarray*}
&&  \nabla  f({x^*})  - \sum_{i = 1}^{l} \lambda_i \nabla g_i({x^*}) - \sum_{j = 1}^{m} \nu_j \nabla h_j({x^*})  = 0, \\
&&  \lambda_i g_i({x^*}) = 0, \ \lambda_i \ge 0, \ \textrm{ for } \ i  = 1, \ldots, l.
\end{eqnarray*}
It follows that $x^*$ a stationary point of the Lagrangian function
$$L(x) := f({x})  - \sum_{i = 1}^{l} \lambda_i g_i({x}) - \sum_{j = 1}^{m} \nu_j h_j({x}) .$$
However, in general, $x^*$ is not a global minimizer of $L$ (and may not even be a local minimizer).

On the other hand, we have the following global optimality conditions, which is inspired by the work of Lasserre~\cite[Chapter~7]{Lasserre2015} (see also \cite[Subsection~7.4.5]{HaHV2017}).

\begin{theorem}\label{GlobalOptimalityConditions}
Assume that $f_* := \min_{x \in S} f(x) > -\infty.$ If the regularity, strict complementarity and second-order sufficiency conditions hold at every minimizers of the restriction of $f$ on $S,$ there are definable $C^{p - 2}$-functions $\phi_i, \psi_j \colon \mathbb{R}^n \rightarrow \mathbb{R}$ for $i = 0, \ldots, l$ and $j = 1, \ldots, m,$ where each $\phi_i$ is a sum of squares of definable $C^{p - 2}$-functions, such that
\begin{eqnarray} \label{EqnLasserre8}
f - f_* &=& \phi_0 + \sum_{i = 1}^l \phi_i g_i + \sum_{j = 1}^m \psi_j h_j.
\end{eqnarray}
Let $x^* \in S$ be a global minimizer of $f$ on $S.$ The following statements hold:
\begin{enumerate}
    \item [{\rm (i)}] $\phi_0(x^*) = 0,$ and $\phi_i(x^*) g_i(x^*) = 0$ and $\phi_i(x^*) \ge 0$ for all $i = 1, \ldots, l.$
    \item [{\rm (ii)}] $\nabla f(x^*) - \sum_{i = 1}^l \phi_i(x^*) \nabla  g_i(x^*) - \sum_{j = 1}^m \psi_j(x^*) \nabla h_j (x^*) = 0.$
    \item [{\rm (iii)}] $x^*$ is a global minimizer of the (generalized) Lagrangian function
    $$\mathscr{L}(x) := f(x) - f_* -  \sum_{i = 1}^l \phi_i(x) g_i(x) - \sum_{j = 1}^m \psi_j(x) h_j(x).$$
\end{enumerate}
\end{theorem}
\begin{proof}
The existence of functions $\phi_i$ and $\psi_j$ for which the equality~\eqref{EqnLasserre8} holds follows immediately from Theorem~\ref{Nichtnegativstellensatz-2}.

(i) From \eqref{EqnLasserre8} and the fact that $x^*$ is a global minimizer of $f$ on $S,$ we get
\begin{eqnarray*}
0 \ = \ f(x^*) - f_* &=& \phi_0(x^*) + \sum_{i = 1}^l \phi_i(x^*) g_i(x^*) + \sum_{j = 1}^m \psi_j(x^*) h_j(x^*),
\end{eqnarray*}
which in turn implies (i) because $g_i(x^*) \ge 0$ for $i = 1, \ldots, l,$ $h_j(x^*)  = 0$ for $j = 1, \ldots, m,$ and the functions $\phi_i$ are all sums of squares of definable $C^{p - 2}$-functions, hence nonnegative. 

(ii) Differentiating \eqref{EqnLasserre8}, using (i) and the fact that the functions $\phi_i$ are sums of squares of definable $C^{p - 2}$-functions yield (ii).

(iii) Since $\phi_0$ is a sum of squares of definable $C^{p - 2}$-functions, we have for all $x \in \mathbb{R}^n,$
$$\mathscr{L}(x) = f(x) - f_* -  \sum_{i = 1}^l \phi_i(x) g_i(x) - \sum_{j = 1}^m \psi_j(x) h_j(x) = \phi_0(x) \ge 0.$$
Note that $\mathscr{L}(x^*) = \phi_0(x^*) = 0.$ Therefore, $x^*$ is a global minimizer of $\mathscr{L}.$ 
\end{proof}

\begin{remark}{\rm 
Theorem~\ref{GlobalOptimalityConditions} implies the following facts.
\begin{enumerate}
\item[{\rm (i)}] The equality~\eqref{EqnLasserre8} can be interpreted as a {\em global optimality condition}. 

\item[{\rm (ii)}] The function $\mathscr{L}(x) := f(x) - f_* -  \sum_{i = 1}^l \phi_i(x) g_i(x) - \sum_{j = 1}^m \psi_j (x)h_j(x)$  
is a {\em generalized Lagrangian function}, with generalized Lagrange (definable functions) multipliers $((\phi_i), (\psi_j))$
instead of scalar multipliers $(\lambda, \nu) \in \mathbb{R}_+^l \times \mathbb{R}^m.$ It is a sum of squares of definable $C^{p - 2}$-functions (hence nonnegative on $\mathbb{R}^n$), vanishes at every global minimizer $x^* \in S,$ and so $x^*$ is also a global minimizer of the generalized Lagrangian function.

\item[{\rm (iii)}] The generalized Lagrange multipliers $((\phi_i), (\psi_j))$ provide a {\em certificate of global optimality} for $x^* \in S$ in the nonconvex case exactly as the Lagrange multipliers $(\lambda, \nu) \in \mathbb{R}_+^l \times \mathbb{R}^m$  provide a certificate in the convex case.
\end{enumerate}

We should also mention that in the KKT~optimality conditions, only the constraints $g_i(x) \ge 0$ that are active at $x^*$ have a possibly nontrivial associated Lagrange (scalar) multiplier $\lambda_i.$  Hence the nonactive constraints do not appear in the Lagrangian function $\mathscr L$ defined above. In contrast, in the global optimality condition \eqref{EqnLasserre8}, every constraint $g_i(x) \ge 0$ has a possibly nontrivial Lagrange multiplier $\phi_i(x).$ But if $g_i(x^*) > 0$ then necessarily $\phi_i(x^*) = 0 = \lambda_i,$ as in the KKT~optimality conditions.
}\end{remark}

\bibliographystyle{abbrv}
%\bibliography{D:/Submission/BibPureMath1,D:/Submission/BibAppMath1}
%\bibliography{H:/Submission/DST-BibPureMath1,H:/Submission/DST-BibAppMath1}

\end{document}